\documentclass[10pt,a4paper]{amsart}
\usepackage[english]{babel}
\usepackage{hyperref}
\usepackage{amssymb}
\usepackage{aliascnt}
\usepackage{xspace}

\DeclareMathOperator{\acc}{acc}
\DeclareMathOperator{\dom}{dom}
\DeclareMathOperator{\otp}{otp}

\DeclareMathOperator{\len}{len}
\DeclareMathOperator{\crit}{crit}

\DeclareMathOperator{\cf}{cf}
\DeclareMathOperator{\stem}{stem}
\DeclareMathOperator{\Col}{Col}
\DeclareMathOperator{\Ult}{Ult}
\DeclareMathOperator{\Lev}{Lev}
\newcommand{\ZFC}{{\rm ZFC}\xspace}
\newcommand{\SCH}{{\rm SCH}\xspace}
\newcommand{\bM}{\mathbb{M}}
\newcommand{\bP}{\mathbb{P}}
\newcommand{\bS}{\mathbb{S}}
\newcommand{\bT}{\mathbb{T}}
\newcommand{\tup}[1]{\left\langle#1\right\rangle}

\newtheorem{theorem}{Theorem}
\newaliascnt{corollary}{theorem}

\aliascntresetthe{corollary}
\newaliascnt{lemma}{theorem}
\newtheorem{lemma}[lemma]{Lemma}
\aliascntresetthe{lemma}
\newaliascnt{claim}{theorem}

\aliascntresetthe{claim}
\theoremstyle{definition}
\newaliascnt{definition}{theorem}
\newtheorem{definition}[definition]{Definition}
\aliascntresetthe{definition}
\newtheorem{question}{Question}

\begin{document}
\author{Yair Hayut}
\email[Yair Hayut]{yair.hayut@mail.huji.ac.il}
\address{Einstein Institute of Mathematics \\
Edmond J. Safra Campus \\
The Hebrew University of Jerusalem \\
Givat Ram. Jerusalem, 9190401, Israel}

\author{Menachem Magidor}
\email[Menachem Magidor]{mensara@savion.huji.ac.il}
\address{Einstein Institute of Mathematics \\
Edmond J. Safra Campus \\
The Hebrew University of Jerusalem \\
Givat Ram. Jerusalem, 9190401, Israel}
\title{Destructibility of the tree property at \texorpdfstring{$\aleph_{\omega+1}$}{aleph omega plus 1}}
\begin{abstract}
We construct a model in which the tree property holds in $\aleph_{\omega + 1}$ and it is destructible under $\Col(\omega, \omega_1)$. On the other hand we discuss some cases in which the tree property is indestructible under small or closed forcings.
\end{abstract}
\maketitle
\section{Introduction}
A partial order $\langle T, \leq_T\rangle$ is called a \emph{tree}, if it has a minimal element and for every $t\in T$, the set $\{s\in T\mid s\leq_T t\}$ is well ordered by $\leq_T$. The order type of the chain of elements that lie below $t$ in the tree order is called the \emph{level} of $t$ and denoted by $\Lev_T (t)$. For a cardinal $\kappa$, $T$ is called a $\kappa$-tree if $\sup_{t\in T} (\Lev_T(t) + 1) = \kappa$ and the cardinality of each level of $T$ is strictly below $\kappa$.

By a theorem of K\H{o}nig, every $\omega$-tree has a cofinal branch (namely, a cofinal chain). On the other hand, a theorem of Aronszajn states that there is an $\omega_1$-tree that has no cofinal branches. Such a tree is called \emph{Aronszajn tree}. For any larger successor cardinal, $\kappa > \omega_1$, it is independent of \ZFC whether there is $\kappa$-tree with no cofinal branches. This question is related to other combinatorial topics and in order to get the consistency of the non-existence of $\kappa$-Aronszajn tree, one must assume the consistency of some large cardinals. If every $\kappa$ tree has a cofinal branch, we say that $\kappa$ has the \emph{tree property}.  

By a theorem of Silver, if uncountable cardinal $\kappa$ has the tree property then $\kappa$ is weakly compact in $L$. On the other end, Mitchell proved that if $\kappa$ is weakly compact and $\mu < \kappa$ is regular then there is a generic extension in which $\kappa = \mu^{++}$ and the tree property holds at $\kappa$, thus showing that the tree property at the double successor of a regular cardinal is equiconsistent with the existence of a weakly compact cardinal. Where $\kappa$ is the successor of a singular cardinal, the situation is more complicated. In \cite{MagidorShelah1996}, Magidor and Shelah showed that it is consistent, relative to some large cardinals, that the tree property holds at $\aleph_{\omega+1}$. The large cardinals assumption was later reduced by Sinapova and Neeman to the existence of an $\omega$-sequence of supercompact cardinals (see, e.g.\ \cite{Neeman2014} for the Prikry-free version). In both constructions, $\aleph_1$ plays a special role. It reflects, in some sense, the properties of $\aleph_{\omega+1}$.

In section \ref{sec:destructible} we will show that it is consistent to have a model in which the tree property holds at $\aleph_{\omega+1}$, but after collapsing $\aleph_1$, it fails. This extends a work by Cummings, Foreman and the second author \cite[Theorem 14]{CummingsForemanMagidor2001}. In this paper they show that it is possible that a weak square is added by a small forcing. Our arguments are very similar to the arguments there. In \cite{Rinot2009}, Rinot shows that it is consistent that there is no special Aronszajn tree on $\aleph_{\omega_1 + 1}$ and a $\sigma$-closed $\aleph_2$-Knaster forcing of cardinality $\aleph_3$ introduces one. We note that we do not know how to apply a similar argument for this case.

In section \ref{sec:indestructible} we discuss three cases in which the tree property at a successor of a singular cardinal is somewhat indestructible. In \ref{subsec: indestructible omega^2} we will show that it is consistent that the tree property holds at $\aleph_{\omega^2+1}$ and it is indestructible under any forcing of cardinality ${<}\aleph_{\omega^2}$. In \ref{subsec: indestructible omega closed} we will show that the tree property at $\aleph_{\omega+1}$ can be made indestructible under small $\sigma$-closed forcings.
\section{Preliminaries}
The following notation, due to Magidor and Shelah \cite{MagidorShelah1996}, plays an important role in the investigation of the tree property at successors of singular cardinals. For more information about narrow systems and their connections to squares we refer to \cite{LambieHanson2017}.
\begin{definition}
Let $\lambda$ be a regular cardinal. A \emph{system} is a triplet $\mathcal{S} = \langle I, \kappa, \mathcal{R}\rangle$ such that:
\begin{enumerate}
\item $I\subseteq \lambda$ unbounded. $\kappa < \lambda$.
\item $\mathcal{R}$ is a collection of partial order relations on $I\times\kappa$.
\item Each $R\in\mathcal{R}$ is a tree like partial order. $R$ respects the lexicographic order on $I\times \kappa$. Namely, $\langle \alpha, \zeta\rangle R \langle \beta, \xi\rangle$ implies $\alpha \leq \beta$ and if $\alpha = \beta$ then $\zeta = \xi$. Moreover, if $\langle \beta, \xi\rangle, \langle \gamma, \rho\rangle R \langle \alpha, \zeta\rangle$ and $\beta \leq \gamma$ then $\langle \beta, \xi\rangle R \langle \gamma, \rho\rangle$.
\item For every $\alpha < \beta$ in $I$ there are $\zeta, \xi < \kappa$ and $R\in\mathcal{R}$ such that $\langle \alpha, \zeta\rangle R \langle \beta, \xi\rangle$.
\end{enumerate}
A \emph{branch} through $\mathcal{S}$ is a set of elements in $I\times \kappa$ which is a chain relative to some $R\in\mathcal{R}$. We say that a branch $b$ meets the $\alpha$-th level of $\mathcal{S}$ if $b\cap \{\alpha\}\times \kappa \neq \emptyset$. A branch is \emph{cofinal} if it meets cofinally many levels.

A system $\mathcal{S}$ is \emph{narrow} if $\max(\kappa^+, |\mathcal{R}|^+) < \lambda$.
\end{definition}
\begin{definition}
Let $\lambda$ be a regular cardinal. We say that the \emph{narrow system property} holds at $\lambda$ if every narrow system of height $\lambda$ has a cofinal branch.
\end{definition}
Unlike the tree property, the narrow system property is indestructible by any small forcing. 
Let $\mathbb{P}$ be a forcing notion with $|\mathbb{P}|^+ < \lambda$ and let $\dot{\mathcal{S}}$ be a name for a narrow system. Let $\dot{\mathcal{R}}$ be the collection of names of relations in $\mathcal{S}$ and let $I$ be the set of all ordinals that can be levels of the $\mathbb{P}$. Let us define the narrow system $\hat{\mathcal{S}}$ in the natural way: the relations of $\hat{\mathcal{S}}$ are indexed by $\mathbb{P}\times\dot{\mathcal{R}}$, and let $\langle \alpha, \beta \rangle (p, R) \langle \gamma, \delta\rangle$ iff $p\Vdash \langle \alpha,\beta\rangle R \langle \gamma ,\delta\rangle$ for $R\in\dot{\mathcal{R}}$. A branch in the system $\hat{\mathcal{S}}$ corresponds to a condition $p\in\mathbb{P}$ and a set of element in $\mathcal{S}$ which are forced to be a branch in the generic extension by $p$.

\section{Destructible tree property}\label{sec:destructible}
\begin{theorem}\label{thm: destructible tree property} Let $\kappa = \kappa_0 < \kappa_1 < \cdots$ be an $\omega$-sequence of supercompact cardinals. Then there is a forcing extension in which the tree property holds at $\aleph_{\omega + 1}$ and the forcing $\Col(\omega, \omega_1)$ adds a special $\aleph_{\omega+1}$-Aronszajn tree.
\end{theorem}
We will prove something slightly stronger. We will define a forcing poset that forces that in the generic extension there is a partial weak square on $\aleph_{\omega + 1}$ whose domain contains all ordinals with cofinality above $\omega_1$, while the tree property holds at $\aleph_{\omega+1}$. If we further extend the universe and collapse $\omega_1$ to be countable, then we can complete all the missing places in this square sequence by just adding $\omega$ sequences. By a theorem of Shelah and Ben-David \cite[Theorem 3]{ShelahBenDavid1986}, without violating the continuum hypothesis at $\aleph_\omega$, we cannot hope to have this kind of partial square with only one club at each ordinal, while having the tree property. 

Let $\mu = \sup \kappa_n$ and let $\lambda = \mu^+$.

We begin with some definitions:
\begin{definition} A partial square on a set $S\subseteq \lambda$ with width $<\eta$ is a sequence $\mathcal{C} = \langle \mathcal{C}_\alpha \mid \alpha < \lambda\rangle$ such that:
\begin{enumerate}
\item For every $\alpha < \lambda$, $\mathcal{C}_\alpha$ is a set of cardinality $<\eta$. If $\alpha \in S$ then $\mathcal{C}_\alpha \neq \emptyset$.
\item Every $D\in\mathcal{C}_\alpha$ is a closed and unbounded subset of $\alpha$ and $\otp D < \alpha$.
\item If $\beta \in \acc D$, $D\in \mathcal{C}_\alpha$ then $D\cap \beta \in \mathcal{C}_\beta$.
\end{enumerate}
\end{definition}
When $\lambda = \mu^+$, we may assume that $\otp D \leq \mu$ for every $D\in\mathcal{C}_\alpha$.

Since successor ordinals are never accumulation points of a club, the values of the square sequence at successor points are irrelevant. We will assume that $\mathcal{C}_{\alpha+ 1} = \{\alpha\}$ for every $\alpha$, for consistency.

We want to force a partial square for the set $S^\lambda_{\geq\kappa}$ with width $<\mu$.
\begin{definition}
Let $\mathbb{S}$ be the following forcing notion. A condition $s\in \mathbb{S}$ is a sequence $s = \langle c_i \mid i \leq \gamma\rangle$ for some ordinal $\gamma < \mu^+$ such that all three requirements for the partial square sequence hold for every $\alpha \leq \gamma$. Namely,
\begin{enumerate}
\item $\forall \alpha \leq \gamma$, $c_\alpha$ is a set of less than $\mu$ sets. If $\cf \alpha \geq \kappa$, then $c_\alpha \neq \emptyset$.
\item For every $D\in c_\alpha$, $\otp D \leq \mu$ and $D$ is a closed and unbounded subset of $\alpha$.
\item If $\beta\in\acc D$, $D\in c_\alpha$ then $D\cap \beta\in c_\beta$.
\end{enumerate}
We order $\mathbb{S}$ by end extension.
\end{definition}
We will think of the conditions $s\in\mathbb{S}$ as functions, so for $s = \langle c_i \mid i \leq \gamma\rangle$ we will write $\dom s = \gamma + 1$ and $s(i) = c_i$ for $i\in\dom s$. 
\begin{lemma}
$\mathbb{S}$ is $\kappa$-directed closed.
\end{lemma}
Given a partial square $\mathcal{C}$, we will define a threading forcing, $\mathbb{T}_\eta$. This forcing will add a club at $\lambda$ with order type $\eta$ such that all its initial segments are from $\mathcal{C}$.

\begin{definition}
Let $\mathbb{T}_\eta = \{D \mid \exists \alpha,\,D\in\mathcal{C}_\alpha,\,1 < \otp D < \eta\}$, ordered by end extension.
\end{definition}
The following lemma is standard:
\begin{lemma} Let $\mathbb{S}, \mathbb{T}_\eta$ be as above. Then:
\begin{enumerate}
\item $\mathbb{S}$ is $\lambda$-distributive.
\item Let $\mathcal{C}$ be the generic partial square added by $\mathbb{S}$, and let $\eta$ be a regular cardinal. $\mathbb{S}\ast\mathbb{T}_{\eta}$ is equivalent to an $\eta$-directed closed forcing. Moreover, for every $\rho < \mu$, $\mathbb{S}\ast\mathbb{T}_{\eta}^\rho$ (where we use full support power in $V^{\mathbb{S}}$) contains an $\eta$-directed closed dense subset.
\end{enumerate}
\end{lemma}
\begin{proof}
Let us show that $\mathbb{S}$ is $\lambda$-distributive. We will show that it is $\eta$-strategically closed for every regular $\eta < \lambda$. We will do this by showing the second part of the lemma -- that $\mathbb{S}\ast\mathbb{T}_\eta$ contains a $\eta$-closed dense set.  

Let us observe first that the set of conditions $\langle s, \check{t}\rangle \in \mathbb{S}\ast\mathbb{T}_\eta$, $\dom(s) = \gamma + 1$, $t\in s(\gamma)$ is dense. For every condition $\langle s, \dot{t}\rangle$, \[s\Vdash \text{``}\dot{t}\text{ is a member of some set in the square sequence''},\] and therefore $\dot{t}$ is forced to be a member of the ground model. 

Thus, there is an extension of $s$, $s^\prime$, which decides the value of $\dot{t}$ to be equal to an element in $V$, that we will denote by $t$. The closed set $t$ might have no extension in $s^\prime(\max \dom s^\prime)$ but we can extend $s^\prime$ to $s^{\prime\prime}$ where $\dom s^{\prime\prime} = \dom s^\prime + \omega + 1$, and $t$ has an extension in the top element of $s^{\prime\prime}$. Let call this extension $t^\prime$. Thus we have a condition $\langle s^{\prime\prime}, t^{\prime}\rangle \leq \langle s, t\rangle$ and $\langle s^{\prime\prime}, t^{\prime}\rangle$ has the desired form.

The set \[D = \{\langle s, \check{t}\rangle \in \mathbb{S}\ast\mathbb{T}_\eta \mid \max t = \max \dom s\}\] is $\eta$-directed closed. Let $\rho < \eta$ and let $\{\langle s_i, \check{t_i}\rangle \mid i < \rho\} \subset D$ be a directed set. Let us assume that $\sup \dom s_i$ is a limit ordinal (otherwise, the sequence is fixed on a tail). The condition $\langle s_\star, t_\star\rangle$, where $t_\star = \bigcup t_i$ and $s_\star = (\bigcup s_i)^\smallfrown \langle \{t_\star\}\rangle$ is a condition in $D$, stronger than $s_i$ for all $i$.

The claim that $\mathbb{S}\ast\mathbb{T}_\eta^{\rho}$ contains a $\eta$-closed dense subset (for all $\rho < \mu$), is proved by the same method. For this case, we consider 
\[D = \{\langle s, \langle t_\alpha \mid \alpha < \rho\rangle \rangle \mid \forall \alpha < \rho,\, \max t_\alpha = \max \dom s\}.\]
By the same argument, using the fact that the bound on the cardinality of the set $s(\max \dom s)$, for $s\in\mathbb{S}$, is greater than $\rho$, we conclude that $D$ is dense and $\eta$-directed closed in $\mathbb{S}\ast \mathbb{T}_\eta^\rho$.  
\end{proof}

Let us move now toward the proof of \ref{thm: destructible tree property}. Let $\kappa_0 < \kappa_1 < \cdots < \kappa_n < \cdots$ be supercompact cardinals. By using Laver's preparation, we may assume that they are Laver-indestructible, i.e.\ that for every $n < \omega$ and every $\kappa_n$-directed closed forcing $\mathbb{P}$, $\Vdash_{\mathbb{P}} \check{\kappa}_n$ is supercompact.
Let $\mathbb{M} = \prod_{i<\omega} \Col (\kappa_i, <\kappa_{i+1})$ a full support product of Levy collapses.
\begin{lemma}\label{lem: nsp after partial square}
After forcing with $\mathbb{S}\times\mathbb{M}$, the narrow system property holds at $\lambda$.
\end{lemma}
\begin{proof}
Let $H_S \subseteq \bS$, $H_M\subseteq \bM$ be mutually generic filters. Let $G = H_S \times H_M$. Let us denote by $H_i\subseteq \Col(\kappa_{i-1}, {<}\kappa_{i})$ be the $i$-th coordinate of the generic filter $H_M$ ($i > 0$). Let $H^i$ be the generic filters for all the parts of $\bM$ except the $i$-th coordinate, namely $H^i = \langle H_m \mid m \neq i\rangle$.

Let $\mathcal{S}\in V[G]$ be a narrow system on $I \times \eta$, with relations $\mathcal{R}$. Let us assume, towards a contradiction, that $\mathcal{S}$ has no cofinal branch in $V[G]$. Since the set $I$ will play no role later in the proof, we will restrict ourselves to the notation-wise simpler case in which $I = \lambda$. Let $n \geq 2$ be large enough such that $\kappa_{n - 2} \geq |\eta \times \mathcal{R}|^+$ in $V^{\mathbb{S}\times \mathbb{M}}$.

Let $W_n = V[H_S][H^n].$ Let us force over $W_n$ with $\bT_{\kappa_n}^{\kappa_{n-2}}$. Let $K = \langle K_i \mid i < \kappa_{n-2}\rangle$ be the sequence of pairwise mutually generic filters. We stress that the product, $\bT_{\kappa_n}^{\kappa_{n-2}}$, is taken over $V[G]$ and not over $W_n$.

Fix $\xi < \kappa_{n-2}$. $W_n[K_\xi]\models \kappa_n$ is supercompact since:
\begin{enumerate}
\item $\bS\ast \bT_{\kappa_n}^{\kappa_{n-2}}$ contains a dense $\kappa_n$-directed closed subset,
\item $\prod_{n\leq i < \omega} \Col(\kappa_i, <\kappa_{i+1})$ is $\kappa_n$-directed closed.
\item $\prod_{i < n - 1} \Col(\kappa_i, <\kappa_{i+1})$ has cardinality $\kappa_{n-1}$ which is $<\kappa_n$.
\end{enumerate}
We are using the indestructibility in the two first items and L\'evy-Solovay Theorem in the last one.

Let $j\colon W_n[K_\xi] \to M$ be a $\lambda$-supercompact embedding with $\crit j = \kappa_n$. Since $\Col(\kappa_{n-1}, <\kappa_n)$ is $\kappa_n$-c.c., after forcing with 
\[\Col(\kappa_{n-1}, <j(\kappa_n)) = \Col(\kappa_{n-1}, <\kappa_n)\times\Col(\kappa_{n-1}, [\kappa_n, j(\kappa_n)))\] 
we may extend the elementary embedding $j$ to a $\lambda$-supercompact elementary embedding $\tilde{j}\colon W_n[H_n][K_\xi] \to M[\tilde{j}(H_n)]$. Since $W_n[H_n] = V[G]$, $S\in W_n[H_n]$, so $\tilde{j}(\mathcal{S})$ is defined.

Let $L = \langle L_i \mid i < \kappa_{n-2}\rangle$ be a generic filter for $\Col(\kappa_{n-1}, [\kappa_n, j(\kappa_n)))^{\kappa_{n-2}}$. Note that the forcing that adds $L$ is $\kappa_{n-1}$-closed over $V$, the ground model.

Let $\delta = \sup \tilde{j}^{\prime\prime} \lambda < \tilde{j}(\lambda)$. Let $\leq_i\in \mathcal{R}$ and let
\[b_{i,\epsilon} = \{\langle \alpha, \beta\rangle \mid \tup{j(\alpha), \beta} \leq_{i} \tup{\delta,\epsilon} \text{ in } \tilde{j}(\mathcal{S})\}.\]
Since $|\mathcal{R}|, \eta < \kappa_{n-2} < \crit \tilde{j}$, for some $i, \epsilon$, $b_{i,\epsilon}$ is a cofinal branch and moreover $\bigcup_{i, \epsilon} \{\alpha \mid \exists \beta,\,\tup{\alpha,\beta}\in b_{i,\epsilon}\} = \lambda$.

We say that forcing with $\Col(\kappa_{n-1}, [\kappa_n, j(\kappa_n))) \times \mathbb{T}_{\kappa_n}$ adds a \emph{system of branches} for $\mathcal{S}$. By removing a bounded part we may assume that all the branches in this system of branches are new and cofinal.

In particular the forcing $\Col(\kappa_{n-1}, [\kappa_n, <j(\kappa_n)) )^{\kappa_{n-2}} \times \bT^{\kappa_{n-2}}_{\kappa_n}$ introduces $\kappa_{n-2}$ many distinct realizations for the system of branches $\{\dot{b_j} \mid j \in J\}$. Note that in order to claim that there is no pair of system of branches which are equal we only used the pairwise mutual genericity.

We conclude that in $V[G][H][K][L]$ there are $\kappa_{n - 2}$ different systems of branches, $\{b_j^\alpha \mid \alpha < \kappa_{n-2},\ j\in J\}$. In this model $\kappa_{n-2} \geq |\eta\times \mathcal{R}|^+$ is regular and $\cf \lambda \geq \kappa_{n-1}$. Since for every $\alpha < \beta < \kappa_{n-2}$, and every relation $\leq_i\in \mathcal{R}$, $b^\alpha_i, b^\beta_i$ split at some point below $\lambda$, and since there are only $\kappa_{n-2}$ realizations and only $|\mathcal{R}|$ relations in $\mathcal{R}$, there is $\rho_\star < \lambda$ such that for every $\xi \geq \rho_\star$, and for every $\alpha, \beta$, $b^\alpha_i(\xi) \neq b^\beta_i(\xi)$ (where it is possible that only one of them is defined). By the Pigeonhole Principle there are $\alpha ,\beta < \kappa_{n-2}$ such that $\langle \rho_\star, \xi\rangle \in b^\alpha_i, b^\beta_i$ for the same $\xi, i$, because there are only $|\mathcal{R}|\times\eta$ many possibilities for this pair. This is a contradiction to the choice of $\rho_\star$. We conclude that it is impossible that there was not cofinal branch in $\mathcal{S}$ in the ground model, as wanted.
\end{proof}
Let $W = V^{\mathbb{S}\times\mathbb{M}}$. Note that $\kappa$ is supercompact in $W$, by the Laver indestructibility of $\kappa$.
\begin{theorem}\label{thm: choice of rho}
There is $\rho < \kappa$ such that forcing with $\Col(\omega, \rho^{+\omega})\times \Col(\rho^{+\omega + 1}, < \kappa)$ over $W$ forces the tree property at $\aleph_{\omega + 1}$. Further collapsing the new $\aleph_1$ introduces a weak square at $\aleph_{\omega+1}$.
\end{theorem}
\begin{proof}
Assume otherwise.
Let $\mathbb{L}_{\rho} = \Col(\omega, \rho^{+\omega})\times \Col(\rho^{+\omega + 1}, < \kappa)$.
For every $\rho < \kappa$, let $\dot{T}_\rho$, be a $\mathbb{L}_{\rho}$-name for an Aronszajn tree at $\lambda$. Since $\kappa$ is supercompact, there is $j\colon W \to M$ such that $^{\lambda}M \subseteq M$. By our assumption, $M$ models that $\Vdash_{j(\mathbb{L})_{\kappa}} \text{``}j(\dot{T})_\kappa$ is an Aronszajn tree''.
Let $\delta = \sup j `` \lambda < j(\lambda)$, and let $t = \tup{\delta, 0}$.

Work in $M$. For every $\alpha < \lambda$, pick a condition $p_\alpha = \langle c_\alpha, q_\alpha\rangle$ such that
\[\exists \zeta < j(\kappa^{+\omega}),\, p_\alpha \Vdash_{j(\mathbb{L}_\kappa)} \langle j(\alpha), \zeta\rangle \leq_{j(\dot{T})_\kappa} \check{t}\]

Let us denote this $\zeta$ by $\zeta_\alpha$. We may pick the conditions $p_\alpha$ in a way that $q_\alpha$ is a decreasing sequence. Since $\lambda$ is regular and $|\Col(\omega, \kappa^{+\omega})| = \kappa^{+\omega} < \lambda$, there is a cofinal set $I\subseteq \lambda$, $n < \omega$ and $c_\star\in \Col(\omega, \kappa^{+\omega})$ such that for every $\alpha \in I$, $c_\alpha = c_\star$ and $\zeta_\alpha < j(\kappa^{+n})$.

By elementarity, for every $\alpha, \beta\in I$, there are $\gamma, \gamma^\prime < \kappa^{+n}$, $\rho < \kappa$ and $p\in \mathbb{L}_\rho$ such that $p\Vdash_{\mathbb{L}_\rho} \langle \alpha, \gamma\rangle \leq_{\dot{T}_\rho}\langle \beta, \gamma^\prime\rangle$.

This defines a narrow system in $W$: The domain of the system is $I \times \kappa^{+n}$. The indices set is $\bigcup_{\rho < \kappa} \mathbb{L}_\rho \times \{\rho\}$. $\langle \alpha, \xi\rangle \leq_{p, \rho} \langle \beta, \zeta\rangle$ iff $p\Vdash_{\mathbb{L}_\rho} \langle \alpha, \xi\rangle \leq_{\dot{T}_\rho} \langle \beta, \zeta\rangle$.

By the narrow system property there is a cofinal branch in $W$. Namely there are $\rho < \kappa$, $p\in\mathbb{L}_\rho$ and $\gamma < \kappa^{+n}$ such that for every $\alpha, \beta\in I$, $p\Vdash_{\mathbb{L}_\rho} \langle \alpha, \gamma\rangle \leq \langle \beta, \gamma\rangle$.

This proves that the tree property holds at $\aleph_{\omega+1}$ in the generic extension.

For the last claim, note that after collapsing $\aleph_1$, for every $\gamma < \aleph_{\omega+1}$ either $\cf \gamma = \omega$ or $\mathcal{C}_\gamma\neq\emptyset$. Thus, one can complete the partial square to a full $\square_{\aleph_{\omega}, <\aleph_{\omega}}$ by adding cofinal $\omega$-sequences.
\end{proof}
\section{Indestructible tree property}\label{sec:indestructible}
In this section we will build three models in which the tree property at a successor of singular cardinal is indestructible under certain class of forcing notions. We start by building a model in which the tree property holds at $\aleph_{\omega^2 + 1}$ and it is indestructible under any forcing $\mathbb{P}$ of cardinality less than $\aleph_{\omega^2}$. Similarly, we will construct a model for the tree property at $\aleph_{\omega + 1}$ in which the tree property still holds after any $\sigma$-closed forcing of cardinality $<\aleph_{\omega}$.

We remark that we do not know whether it is possible to force the tree property at $\aleph_{\omega+1}$ to be indestructible under any $\aleph_{\omega+1}$-closed forcing notions.

\subsection{Indestructible Tree Property for \texorpdfstring{$\aleph_{\omega^2 + 1}$}{aleph omega2 + 1}}\label{subsec: indestructible omega^2}
In this subsection, we will show that in Sinapova's model for the tree property at $\aleph_{\omega^2+1}$ \cite{Sinapova2012_aleph_omega2} (but without the failure of $\SCH$, as in \cite{Sinapova2012aleph_omega}), the tree property is indestructible under small forcings. We start with some simple observations:
\begin{lemma}
Let $\lambda$ be a cardinal such that the tree property holds at $\lambda^+$ and it is indestructible by any forcing of the form $\Col(\omega, \rho)$ for $\rho < \lambda$. Then the tree property at $\lambda^+$ is indestructible by any forcing of size $<\lambda$. Moreover, it is enough to assume that for every $\rho < \lambda$ there is $\rho\leq \rho^\prime < \lambda$ such that $\Col(\omega,\rho^\prime)$ forces the tree property at $\lambda^+$.
\end{lemma}
\begin{proof}
Let $\mathbb{P}$ be a forcing notion of cardinality $<\lambda$. Let $\mu = |\mathbb{P}|$. $\Col(\omega, \rho)$ adds a generic filter for $\mathbb{P}$. Let $G\subseteq \mathbb{P}$ be a generic filter. The quotient forcing $\Col(\omega, \rho)/G$ has cardinality at most $\rho$ and therefore it does not add a cofinal branch to any $\lambda^+$-Aronszajn tree. Since the tree property holds after forcing with $\Col(\omega,\rho)$ and the forcing $\Col(\omega, \rho)/G$ does not add a branch to Aronszajn tree -- the tree property holds in $V[G]$ as well.
\end{proof}
\begin{theorem}\label{thm: indestructible tree property}
Let $\kappa = \kappa_0 < \kappa_1 < \cdots$ be a sequence of $\omega$ supercompact cardinals. Let $\mu = \sup \kappa_n$ and $\lambda = \mu^+$. There is a generic extension in which $\kappa = \aleph_{\omega^2}$, $\lambda = \aleph_{\omega^2 +1}$ and for every $\rho < \mu$, the tree property holds after forcing with $\Col(\omega, \rho)$.
\end{theorem}
In order to prove this theorem, we will work with Sinapova's model for the tree property at $\aleph_{\omega^2+1}$ from \cite{Sinapova2012_aleph_omega2}. We will not need to violate \SCH at this point, so the proof is somewhat simpler at some points.

The main idea behind the indestructibility is that one can define a projection $f\colon \bP \times \Col(\omega, \rho) \to\bP_n$ that shifts the Prikry sequence by $n$ steps to the left, where $\bP_n$ is a ``shifted'' version of the forcing $\bP$ which forces the tree property as well. This way, we can analyze the sets that were added by a forcing of the form $\Col(\omega, \rho)$ simply by shifting the first element of the Prikry sequence to be above $\rho$. 

We start with a well known fact:
\begin{lemma}\label{lem: nsp with collapses}
Let $\bM = \prod_{n < \omega} \Col(\kappa_n, < \kappa_{n+1})$ - a full support product of Levy collapses.
In $V^{\bM}$ the narrow branch property holds at $\lambda^+$.
\end{lemma}
The proof is similar to the proof of Lemma~\ref{lem: nsp after partial square} and appears in \cite{Neeman2014}. 

Work in $V^{\mathbb{M}}$. The cardinal $\kappa = \kappa_0$ is still supercompact, by the Laver indestructibility. Let $\mathcal{U}$ be a normal measure on $P_\kappa \lambda$ in $V^{\mathbb{M}}$. Let $\mathcal{U}_n$ be the projection of $\mathcal{U}$ to $P_\kappa \kappa_n$ for $n < \omega$.

Let $j_n\colon W\to N_n \cong \Ult(W, \mathcal{U}_n)$ be the elementary embedding derived from $\mathcal{U}_n$. Let us construct an $N_n$-generic filter $H_n$ for the forcing $\Col(\kappa^{+\omega + 2}, < j(\kappa))^{N_n}$. This is possible by the standard arguments: the forcing notion $\Col(\kappa^{+\omega + 2}, < j(\kappa))^{N_n}$ is $\kappa^{+n+1}$-closed in $W$ and has only $\kappa^{+n+1}$-dense subsets in $N_n$ (as counted by $V^{\mathbb{M}}$).

Let us define the main forcing notion $\mathbb{P}$:

A condition $p\in \mathbb{P}$ has the following form
\[p = \langle d_0, a_0, c_0, \dots, a_{n-1}, c_{n-1}, A_n, C_n, \dots\rangle\]
where:
\begin{enumerate}
\item $a_i\in P_\kappa \kappa^{+i}$ and $A_i \in \mathcal{U}_i$. Let $\rho_i = a_i\cap \kappa$ if $i < n$ and $\rho_i = \kappa$ otherwise.
\item $d_0 \in \Col(\omega, \rho_0^{+\omega})$ if $\rho_0 < \kappa$ and otherwise $d_0\in \Col(\omega, \kappa)$.
\item $c_i \in \Col(\rho_i^{+\omega + 2}, <\rho_{i+1})$
\item $C_i\colon A_i \to W$ such that $C_i(a) \in \Col((a\cap \kappa)^{+\omega + 2}, < \kappa)$ for every $a\in A_i$ and $[C_i]_{\mathcal{U}_i}\in H_i$.
\end{enumerate}
$n$ is called the length of $p$ and we denote $\len(p)=n$.

A condition $p$ is stronger than $q$ ($p \leq q$) if:
\begin{enumerate}
\item $\len(p) \geq \len(q)$
\item $d^p_0 \leq d^q_0$.
\item $a_i^p = a_i^q$ and $c^p_i \leq c^q_i$ for every $i < \len(q)$.
\item $a_i^p \in A^q_i$ and $c_i^p \leq C_i^q(a_i)$ for $\len(q)\leq i < \len(p)$
\item $A_i^p \subseteq A_i^q$ for $i\geq\len(p)$.
\item $C^p_i(a) \leq C^q_i(a)$ for every $a\in A^p_i$.
\end{enumerate}

For the proof of Theorem \ref{thm: indestructible tree property}, we will also need to consider the following shifted version of $\mathbb{P}$. For every $s < \omega$, we define the forcing $\mathbb{P}_s$.

A condition $p\in \mathbb{P}_s$ has the following form \[p = \langle d_0, a_0, c_0, \dots, a_{n-1}, c_{n-1}, A_n, C_n, \dots\rangle\] where:
\begin{enumerate}
\item $a_i\in P_\kappa \kappa^{+ i + s}$ and $A_i \in \mathcal{U}_{i + s}$. Let $\rho_i = a_i\cap \kappa$ if $i < n$ and $\rho_i = \kappa$ otherwise.
\item $d_0 \in \Col(\omega, \rho_0^{+\omega})$ if $\rho_0 < \kappa$ and otherwise $d_0\in \Col(\omega, \kappa)$.
\item $c_i \in \Col(\rho_i^{+\omega + 2}, <\rho_{i+1})$
\item $C_i\colon P_\kappa \kappa^{+ i + s} \to W$ such that $C_i(a) \in \Col((a\cap \kappa)^{+\omega + 2}, < \kappa)$ for every $a\in A_i$ and $[C_i]_{\mathcal{U}_{i + s}}\in H_{i + s}$.
\end{enumerate}
We order the conditions in the same way as we did for $\mathbb{P}$. Note that $\mathbb{P}_0 = \mathbb{P}$.

\begin{theorem}[Sinapova]\label{thm: tree property at aleph omega^2 + 1}
For every $s < \omega$, $\mathbb{P}_s$ forces that $\lambda = \aleph_{\omega^2 + 1}$ and the tree property holds in $\lambda$.
\end{theorem}
\begin{proof}
We will give a sketch of the proof. We will show that the claim holds for $s = 0$. The argument for general $s$ is the same, by notation-wise more complicated.

Let $p\in\mathbb{P}$ be a condition and let $\dot{T}$ be a name for a $\lambda$-Aronszajn tree. Let $n$ be the length of $p$. Let $j\colon V\to M$ be a $\lambda$-supercompact embedding, with critical point $\kappa$ which is compatible with $\mathcal{U}_n$ (namely $\mathcal{U}_n$ is the $P_\kappa \kappa^{+n}$ measure which is derived from $j$).

In $M$, let us look at the forcing $j(\mathbb{P})$ below a condition $q\leq j(p)$ of length $n + 1$ such that $a_{n}^q = j^{\prime\prime}\kappa^{+n}$. In other words, $q$ is an extension of $j(p)$ that forces that the $n+1$-th element of the diagonal Prikry sequence is $j^{\prime\prime}\kappa^{+n}$. The forcing $j(\mathbb{P}) / q$ preserves $\lambda$ as a regular cardinal and realizes $j(\dot{T})$ to be a $j(\lambda)$-Aronszajn tree.

Let us denote $\delta = \sup j ``\lambda < j(\lambda)$ and let us look at the name of a partial branch $\{\langle j(\alpha), \zeta_\alpha\rangle \mid M^{j(\mathbb{P})} \models \langle j(\alpha), \zeta_\alpha\rangle \leq_{j(\dot{T})} \langle \delta, 0\rangle\}$.

Using the Prikry property, we may find a direct extension of $q$, $q^\star$, such that for every $\alpha < \lambda$ the value of $k < \omega$ such that $\zeta_\alpha < j(\kappa^{+k})$ is determined by $q^\star$ up to forcing with the first $n$ lower parts of $j(\mathbb{P})$ ($n < \omega$). Since there are less than $\lambda$ many possible values for the first $n$ coordinates of the conditions below $q^\star$, there is a cofinal subset of $\lambda$, $I$, a natural number $n_\star < \omega$ large enough and a fixed lower part $a_\star$ of length $n_\star \geq n + 1$ such that
\[I = \{\alpha < \lambda\mid \exists r \leq q^\star, \stem(r) = a_\star, r\Vdash \exists \zeta < j(\kappa^{+n_\star}), \langle j(\alpha) ,\zeta\rangle \leq \langle\delta, 0\rangle\}.\]

In particular, for every $\alpha, \beta\in I$, $M$ thinks that there is an extension of $j(p)$, $q^{\star\star}$ of length $n+1$ and ordinals $\zeta, \zeta^\prime < j(\kappa^{+n_\star})$ such that $q^{\star\star}\Vdash \langle j(\alpha), \zeta\rangle \leq_{j(\dot{T})} \langle j(\beta), \zeta^\prime\rangle$. Reflecting this to $V$ we conclude that for every $\alpha, \beta\in I$ there is a condition $q'\leq p$ with stem of length $n+1$ and $\zeta, \zeta^\prime < \kappa^{+n_\star}$ such that $q'\Vdash \langle \alpha, \zeta\rangle \leq_{\dot{T}} \langle \beta, \zeta^\prime\rangle$.

This defines a narrow system on $I\times \kappa^{+n_\star}$, indexed by the stems of length $n+1$ which are stems of some condition which is stronger than $p$. By the narrow system property, there is a cofinal branch. So there is $I^\prime \subseteq I$, a stem $s_\star$ and an ordinal $\zeta_\star < \kappa^{+n_\star}$ such that for every $\alpha < \beta$ in $I^\prime$ there is a condition $q$ with stem $s_\star$ forcing $\langle \alpha, \zeta_\star\rangle \leq_{\dot{T}} \langle \beta, \zeta_\star\rangle$.

Next we will build inductively a sequence of conditions $\langle p_\alpha \mid \alpha \in I^\prime \setminus \rho\rangle$ (for some $\rho < \lambda$), such that for every $\alpha < \beta$, \[p_\alpha \wedge p_\beta \Vdash \langle \alpha, \zeta_\star\rangle \leq_{\dot{T}} \langle \beta, \zeta_\star\rangle\]
The construction is done by induction on $m < \omega$, where at each step we define $p_\alpha \restriction m$ in a way that for all $\alpha, \beta$ (except a bounded segment) there is a condition $q$ with $q\restriction m = p_\alpha \restriction m \wedge p_\beta \restriction m$ such that
\[q \Vdash \langle \alpha, \zeta_\star\rangle \leq_{\dot{T}} \langle \beta, \zeta_\star\rangle.\]
Extending $p_\alpha\restriction m$ to $p_\alpha\restriction (m+1)$ is done by defining a narrow system corresponding to the possible extension and using the branch in order to define the relevant value for all $\alpha\in I^\prime$ above the first level that the branch meets.

Eventually, we obtain a sequence of conditions $\{p_\alpha \mid \alpha \in I^\prime\setminus \rho\}$, for some $\rho < \lambda$, $p_\alpha \leq p$. Using the chain condition of the forcing $\mathbb{P}$ we conclude that there is an extension of $p$ that forces that for unbounded many ordinals $\alpha < \lambda$, $p_\alpha$ will be in the generic filter. But then $\{\langle \alpha, \zeta_\star\rangle \mid p_\alpha \in G\}$ is a cofinal branch in $\dot{T}$ (where $G$ is the generic filter for $\mathbb{P}$).
\end{proof}

In order to show the indestructibility, we need to show that there is a simple connection between the different shifts of the forcing:
\begin{lemma}\label{lemma: isomorphism with shift and collapse}
Let $p\in \mathbb{P}$, $\len(p) = n + 1$, $n\geq 1$ and let $f\in \Col(\omega, \rho_{n}^{+\omega})$. There is a condition $q \in \mathbb{P}_n$, of length one such that $\rho^q_0 = \rho^p_n$, such that
$\mathbb{P}_n/q\cong\left(\mathbb{P}/p\right) \times \left(\Col(\omega, \rho_{n}^{+\omega}) / f\right)$.
\end{lemma}
\begin{proof}
Let $\eta = \left(\rho^p_n\right)^{+\omega}$. 

The forcing $\mathbb{P}/p$ is the product $\mathbb{C} /p^{< n} \times \mathbb{P}^{\geq n} / p^{\geq n}$ where
\[\mathbb{C} = \Col(\omega, \left(\rho^p_0\right)^{+\omega}) \times \prod_{i < n} \Col(\left(\rho^p_i\right)^{+\omega + 2}, <\rho^p_{i+1})\] 
and $\mathbb{P}^{\geq n}$ is the set of the $n$-upper part of the conditions of $\mathbb{P}$. 
More precisely, a condition $s\in \mathbb{P}^{\geq n}$ is an $\omega$-sequence of the form 
\[s=\langle a^s_{n}, c^s_{n}, \dots a^s_{l-1}, c^s_{l-1}, A^s_l, C^s_l, \dots \rangle,\] where $l\geq n$ and $a^s_i, c^s_i, A^s_i, C^s_i$ are as in the definition of $\mathbb{P}$ (in particular, $a^s_i\in P_\kappa \kappa^{+i}$). 

The conditions $p^{\geq n}\in \mathbb{P}^{\geq n}, p^{< n}\in\mathbb{C}$ are defined as follows: 
\[p^{< n} = \langle d_0^p, c_0^p, \dots, c_{n-1}^p\rangle,\]
\[p^{\geq n} = \langle a^p_n, c^p_n, A^p_{n+1}, C^p_{n+1}, \dots, A^p_{l}, C^p_{l}, \dots\rangle.\]

Clearly, $|\mathbb{C}| \leq \eta$ and thus $\left(\mathbb{C} / p^{<n}\right) \times \left(\Col(\omega, \eta) / f\right) \cong \Col(\omega, \eta)$. 
Let us fix an isomorphism $\pi_0 \colon \Col(\omega, \eta) \to \left(\mathbb{C} / p^{<n}\right) \times \left(\Col(\omega, \eta) / f\right)$. Note that $\pi_0(\emptyset) = (p^{< n}, f)$.

Let $q\in \mathbb{P}_n$ be the condition $\langle \emptyset\rangle ^\smallfrown p^{\geq n}$. 

By the definition of $\mathbb{P}_n$ and $\mathbb{P}^{\geq n}$,
\[\mathbb{P}_n / q \cong \Col(\omega,\eta) \times \left(\mathbb{P}^{\geq n} / p^{\geq n}\right).\]
Combining this with the isomorphism $\pi_0$, we obtain the isomorphism: 
\[\mathbb{P}_n / q \cong \left(\mathbb{P} / p \right)\times \left(\Col(\omega,\eta)/f\right).\]
\end{proof}

\begin{theorem}
$\mathbb{P}$ forces the tree property at $\aleph_{\omega^2+1}$ to be indestructible by any forcing of size $<\aleph_{\omega^2}$.
\end{theorem}
\begin{proof}
Is it enough to show that it is the case for $\Col(\omega, \aleph_{\omega\cdot n})$. Recall that $\aleph_{\omega\cdot n} = \rho_n^{+\omega}$ so we are in the situation of Lemma~\ref{lemma: isomorphism with shift and collapse}. This means that after forcing with $\Col(\omega, \aleph_{\omega\cdot n})$ the tree property holds, as the iteration is isomorphic to the forcing notion $\mathbb{P}_n$ below some condition. 
\end{proof}

\subsection{Indestructible Tree property for \texorpdfstring{$\aleph_{\omega + 1}$ under small \texorpdfstring{$\sigma$}{sigma}-closed forcings}{aleph omega + 1}}\label{subsec: indestructible omega closed}
Let us construct a model very similar to \autoref{subsec: indestructible omega^2}, in which we have the tree property at $\aleph_{\omega + 1}$ and it will be indestructible under any $\sigma$-closed forcing of cardinality $<\aleph_\omega$. The additional restriction on the forcing notions (namely that the forcing is $\sigma$-closed), implies that those forcing notions cannot collapse $\omega_1$.
\begin{theorem}
It is consistent, relative to the existence of $\omega$ many supercompact cardinals, that the tree property holds at $\aleph_{\omega + 1}$ and it is indestructible under any $\sigma$-closed forcing of cardinality $<\aleph_{\omega}$.
\end{theorem}
\begin{proof}
We will start with a model of the narrow system property at $\kappa^{+\omega+1}$ for $\kappa$ a supercompact cardinal. This can be obtained, for example, by forcing with the product of the Levy collapses between the supercompact cardinals as in Lemma~\ref{lem: nsp with collapses}. Let $\mathcal{U}_0$ be a normal ultrafilter on $\kappa$ generated from a $\kappa^{+\omega + 1}$-supercompact elementary embedding, $j\colon V\to M$.

Let us show that for every $n < \omega$, there is a large set $A_n\in \mathcal{U}_0$ such that for every $\rho \in A_n$, forcing with $\mathbb{L}_\rho = \Col(\omega, \rho^{+\omega})\times \Col(\rho^{+\omega + 1}, \kappa^{+n})$ forces the tree property at $\kappa^{+\omega + 1}$.

Assume that this is not the case and let $\dot{T}_\rho$ be a counter example for every bad choice of $\rho$, for a fixed $n < \omega$. Since the set of bad choices is in $\mathcal{U}_0$, $\kappa$ is a bad choice of ordinal in $M$. Let us force with $j(\mathbb{L})_\kappa$, and let $M[H]$ be the generic extension. Let $T = j(\dot{T})_\kappa^H$ be an Aronszajn tree at $j(\kappa^{+\omega + 1})$. Let $\delta = \sup j`` \kappa^{+\omega + 1}$ and for every $\alpha < \kappa^{+\omega + 1}$ let $\beta_\alpha < j(\kappa^{+\omega})$ be the element in the level $j(\alpha)$ below $\langle \delta, 0\rangle$.

Using the same arguments as in the proof of Theorem~\ref{thm: choice of rho}, there is a cofinal set $I \subseteq \kappa^{+\omega + 1}$, a decreasing sequence of conditions $q_\alpha\in \Col(\kappa^{+\omega + 1}, j(\kappa)^{+n})$, a condition $p\in \Col(\omega, \kappa^{+\omega})$ and a natural number $N < \omega$ such that for every $\alpha \in I$ there is $\beta < j(\kappa^{+N})$ such that $(p,q_\alpha)\Vdash \langle j(\alpha), \beta\rangle \leq_{T} \langle \delta, 0\rangle$.

Reflecting this back to $V$, we conclude that for every $\alpha, \alpha^\prime \in I$: $$\exists \beta, \beta^\prime < \kappa^{+N},\ \rho < \kappa,\, p\in \mathbb{L}_\rho\text{ such that }p\Vdash_{\mathbb{L}_\rho} \langle \alpha, \beta\rangle \leq_{T_\rho} \langle \alpha^\prime, \beta^\prime\rangle.$$

This gives us a narrow system, similar to the one in the proof of Theorem~\ref{thm: choice of rho}. A branch through this system provides us an ordinal $\rho$ which was a bad choice, a condition $r\in\mathbb{L}_\rho$, a cofinal set $J\subseteq I$ and for all $\alpha\in J$ an ordinal $\beta_\alpha < \kappa^{+N}$ such that for all $\alpha, \alpha^\prime \in J$, $$r\Vdash_{\mathbb{L}_\rho} \langle \alpha, \beta_\alpha\rangle, \langle \alpha^\prime, \beta_{\alpha^\prime}\rangle \text{ are compatible.}$$
This is a contradiction to the fact that this $\dot{T_\rho}$ was a name for an $\lambda$-Aronszajn tree.

Let $A = \bigcap_{n < \omega} A_n$ and let $\rho\in A$. Forcing with $\Col(\omega, \rho^{+\omega})\times \Col(\rho^{+\omega + 1}, \kappa)$ forces the tree property. For every small $\sigma$-closed forcing notion $\mathbb{Q}$ there is $n$ such that $\Col(\rho^{+\omega + 1}, \kappa) \ast \mathbb{Q}$ is a regular subforcing of $\Col(\rho^{+\omega + 1}, \kappa^{+n})$ and since the tree property holds after this forcing and since the quotient is small and thus cannot add branches to Aronszajn trees - we are done.
\end{proof}
\section{Open questions}
In Section~\ref{subsec: indestructible omega^2} we proved that the tree property at $\aleph_{\omega^2+1}$ can be made indestructible under any small forcing poset.
\begin{question}
Is it consistent that the tree property at $\aleph_{\omega+1}$ is indestructible under any forcing of cardinality $<\aleph_{\omega}$?
\end{question}
On the other hand, one can ask whether it is possible to extend the results of Theorem~\ref{thm: destructible tree property}.
\begin{question}\label{question: destructible with preserving cardinals}
Is it consistent that the tree property holds at $\aleph_{\omega+1}$ but there is a small forcing (of cardinality $<\aleph_{\omega}$), that does not collapse cardinals and adds an $\aleph_{\omega+1}$-Aronszajn tree?
\end{question}
Note that in all the currently known models for the tree property at $\aleph_{\omega+1}$, adding a single Cohen real does not add an Aronszajn tree at $\aleph_{\omega+1}$. So we ask the following stronger version of Question~\ref{question: destructible with preserving cardinals}:
\begin{question}
Is it consistent that the tree property holds at $\aleph_{\omega+1}$ but adding a Cohen real adds an $\aleph_{\omega+1}$-Aronszajn tree?
\end{question}
This question is particularly interesting when we assume that $\aleph_{\omega}$ is strong limit since then adding a Cohen real cannot add a weak square for $\aleph_{\omega}$, assuming that there is no weak square in the ground model.

\section{Acknowledgments}
We would like to thank the anonymous referee for improving the readability and accuracy of this paper.

\providecommand{\bysame}{\leavevmode\hbox to3em{\hrulefill}\thinspace}
\providecommand{\MR}{\relax\ifhmode\unskip\space\fi MR }
\providecommand{\MRhref}[2]{%
  \href{http://www.ams.org/mathscinet-getitem?mr=#1}{#2}
}
\providecommand{\href}[2]{#2}

\end{document}